\documentclass[10pt]{article}
\usepackage{amsmath,amscd,amsthm,amssymb}

 \newtheorem{thm}{Theorem}[section]
 \newtheorem{cor}[thm]{Corollary}
 \newtheorem{lem}[thm]{Lemma}
 \newtheorem{prop}[thm]{Proposition}
 \theoremstyle{definition}
 \newtheorem{defn}[thm]{Definition}
 \theoremstyle{remark}
 \newtheorem{rem}[thm]{Remark}
 \newtheorem*{ex}{Example}
 \numberwithin{equation}{section}

\def\Rn{{\mathbb{R}^n}}

\def\i{\infty}

\def\L1loc{L_{\Phi}^{\rm loc}(\Rn)}

\def\dual{\,^{^{\complement}}\!}

\usepackage{color}

\newcommand{\es}{\mathop{\rm ess \; inf}\limits}

\begin{document}

\begin{center}
\Large \bf Calder\'{o}n-Zygmund operators and their commutators on generalized weighted Orlicz-Morrey spaces
\end{center}

\

\centerline{\large F. Deringoz$^{a}$, V.S. Guliyev$^{b,c,d}$, M.N. Omarova$^{c,e}$, M.A. Ragusa$^{b,f,}$\footnote{
{correspondent author : maragusa@dmi.unict.it (M. Ragusa)
\\
The research of V. Guliyev and M. Ragusa was partially supported by the Ministry of Education and Science of the Russian Federation (Agreement number: 02.a03.21.0008).
}
\\
E-mail: vagif@guliyev.com (V. Guliyev), deringoz@hotmail.com (F. Deringoz), mehribanomarova@yahoo.com (M. Omarova).}}

\

\centerline{$^{a}$ \small  \it Department of Mathematics, Ahi Evran University, Kirsehir, Turkey}

\centerline{$^{b}$\it \small RUDN University, 6 Miklukho-Maklaya St., Moscow,
Russian Federation 117198}

\centerline{$^{c}$\it \small Institute of Mathematics and Mechanics of NAS of Azerbaijan, AZ1141 Baku, Azerbaijan}

\centerline{$^{d}$\it Institute of Applied Mathematics, Baku State University, AZ 1148 Baku, Azerbaijan}

\centerline{$^{e}$\it Baku State University, AZ1141 Baku, Azerbaijan}

\centerline{$^{f}$\it Dipartimento di Matematica, Universit\'{a} di Catania, Catania, Italy}

\

\begin{abstract}
In this paper, we obtain the necessary and sufficient conditions for the weak/strong boundedness of the Calder\'{o}n-Zygmund operators in generalized weighted Orlicz-Morrey spaces.
We also study the boundedness of the commutators of Calder\'{o}n-Zygmund operators on these spaces. Moreover, the boundedness of Calder\'{o}n-Zygmund operators in the vector-valued setting is given.
\end{abstract}

\



\

\section{Introduction}

The classical Morrey spaces were introduced by
Morrey \cite{Morrey} to study the local behavior of solutions to second-order
elliptic partial differential equations. Moreover, various Morrey spaces are
defined in the process of study. Mizuhara \cite{Miz} and Nakai \cite{Nakai}
introduced generalized Morrey spaces $M^{p,\varphi}(\Rn)$ (see, also \cite{GulJIA}); Komori and Shirai \cite{KomShi} defined weighted Morrey spaces $L^{p,\kappa}(w)$; Guliyev \cite{GulEMJ2012} gave a concept of the generalized weighted Morrey spaces $M^{p,\varphi}_{w}({\mathbb R}^n)$ which could be viewed as extension of both $M^{p,\varphi}(\Rn)$ and $L^{p,\kappa}(w)$. In \cite{GulEMJ2012}, the boundedness of the classical operators and their commutators in spaces $M^{p,\varphi}_{w}$ was also studied,
see also \cite{GulKarMustSer, GulAliz, GulOm, GulHamz, KarGulSer}.

The spaces $M^{p,\varphi}_{w}({\mathbb R}^n)$ defined by the norm
\begin{equation*}
\|f\|_{M^{p,\varphi}_{w}} \equiv
\sup\limits_{x\in\Rn, r>0} \varphi(x,r)^{-1} \,
w(B(x,r))^{-1/p} \, \|f\|_{L^{p}_{w}(B(x,r))},
\end{equation*}
where the function $\varphi$ is a positive measurable function on
${\mathbb R}^n \times (0,\infty)$
and
$w$ is a non-negative measurable function on ${\mathbb R}^n$.
Here and everywhere in the sequel $B(x,r)$ is the ball in $\Rn$ of radius $r$ centered
at $x$ and  $|B(x,r)|=v_n r^n$ is its Lebesgue measure, where $v_n$ is the volume of the unit ball in $\Rn$.

The Orlicz spaces $L^{\Phi}$ were first introduced by Orlicz in \cite{456s, 456sb} as generalizations
of Lebesgue spaces $L^p$. Since then, the theory of Orlicz spaces themselves has been well developed and the spaces have been widely used in probability, statistics, potential theory, partial differential equations, as well as harmonic analysis and some other fields of analysis.

In \cite{DerGulSam}, the generalized Orlicz-Morrey space $M^{\Phi,\varphi}(\Rn)$ was introduced to unify Orlicz and generalized Morrey spaces. Other definitions of generalized Orlicz-Morrey spaces can be found in \cite{Nakai0} and \cite{SawSugTan}. In words of \cite{GulHasSawNak}, our generalized Orlicz-Morrey space is the third kind and the ones in \cite{Nakai0} and \cite{SawSugTan} are the first kind and the second kind, respectively. According to the examples in \cite{GalaSawTan}, one can say that the generalized Orlicz-Morrey space of the first kind and the second kind are different. Notice that the definition of the space of the third kind relies only on the fact that $L^{\Phi}$ is a normed linear space, which is independent of the condition that it is generated by modulars. On the other hand, the spaces of the first and the second kind are defined via the family of modulars.

As based on the results of \cite{AkbGulMust}, the following conditions were introduced in \cite{DerGulSam}
(see, also \cite{GulDerHasJIA}) for the boundedness of the maximal operators and the singular integral operators on $M^{\Phi,\varphi}(\Rn)$, respectively,
\begin{equation}\label{eq3.6.VZMaxGO}
\sup_{r<t<\infty} \Phi^{-1}\big(t^{-n}\big) \es_{t<s<\infty}\frac{\varphi_1(x,s)}{\Phi^{-1}\big(s^{-n}\big)} \le C \, \varphi_2(x,r),
\end{equation}
\begin{equation}\label{eq3.6.VZint}
\int_{r}^{\infty} \Big(\es_{t<s<\infty}\frac{\varphi_1(x,s)}{\Phi^{-1}\big(s^{-n}\big)}\Big) \, \Phi^{-1}\big(t^{-n}\big)\frac{dt}{t}  \le C \, \varphi_2(x,r),
\end{equation}
where $C$ does not depend on $x$ and $r$. It was also shown in \cite{DerGulSam}, the condition \eqref{eq3.6.VZMaxGO} is weaker than \eqref{eq3.6.VZint}.

Various versions of generalized weighted Orlicz-Morrey spaces were introduced in \cite{LiNaYaZh}, \cite{KwokTokMJ}, \cite{PoelTor} and \cite{GulOmSaw}. The spaces in \cite{LiNaYaZh} and \cite{KwokTokMJ} can be seen as the weighted version of generalized Orlicz-Morrey spaces of the first kind and the spaces in \cite{PoelTor} can be seen as the weighted version of generalized Orlicz-Morrey spaces of the second kind.
We used the definition of \cite{GulOmSaw} which can be seen as the weighted version of generalized Orlicz-Morrey spaces of the third kind.

In this paper, we shall investigate the boundedness of the singular integral operators and their commutators on generalized weighted Orlicz-Morrey spaces.

By $A \lesssim B$ we mean that $A \le C B$ with some positive constant $C$
independent of appropriate quantities. If $A \lesssim B$ and $B \lesssim A$, we
write $A\approx B$ and say that $A$ and $B$ are  equivalent.

\

\section{Definitions and Preliminary Results}

Even though the $A_p$ class is well known, for completeness, we offer the definition of $A_p$ weight functions. Let $\mathcal{B}=\{B(x,r):x\in\Rn, ~r>0\}$.
\begin{defn}
For, $1<p<\infty$, a locally integrable function $w:\Rn\to [0,\infty)$ is said to be an $A_p$ weight if
$$
\sup_{B\in \mathcal{B}}\left(\frac{1}{|B|}\int_{B}w(x)dx\right)\left(\frac{1}{|B|}\int_{B}w(x)^{-\frac{p^{\prime}}{p}}dx\right)^{\frac{p}{p^{\prime}}}<\infty.
$$
A locally integrable function $w:\Rn\to [0,\infty)$ is said to be an $A_1$ weight if
$$
\frac{1}{|B|}\int_{B}w(y)dy\le Cw(x),\qquad a.e.~~x\in B
$$
for some constant $C>0$. We define $A_{\infty}=\bigcup_{p\geq 1} A_p$.
\end{defn}

For any $w\in A_{\infty}$ and any Lebesgue measurable set $E$, we write $w(E)=\int_{E}w(x)dx$.

We recall the definition of Young functions.

\begin{defn}\label{def2}
A function $\Phi : [0,\i) \rightarrow [0,\i]$ is called a Young function,
if $\Phi$ is convex, left-continuous,
$\lim\limits_{r\rightarrow 0^+} \Phi(r) = \Phi(0) = 0$
and
$\lim\limits_{r\rightarrow \i} \Phi(r) = \i$.
\end{defn}
The convexity and the condition $\Phi(0) = 0$
force any Young function to be increasing.
In particular,
if there exists $s \in (0,\i)$ such that $\Phi(s) = \i$,
then it follows that $\Phi(r) = \i$ for $r \geq s$.

Let $\mathcal{Y}$ be the set of all Young functions $\Phi$ such that
\begin{equation*}
0<\Phi(r)<\i\qquad \text{for} \qquad 0<r<\i.
\end{equation*}
If $\Phi \in \mathcal{Y}$,
then $\Phi$ is absolutely continuous on every closed interval in $[0,\i)$
and bijective from $[0,\i)$ to itself.

For a Young function $\Phi$ and $0 \leq s \leq \i$, let
$$
\Phi^{-1}(s)
\equiv
\inf\{r\geq 0: \Phi(r)>s\}\qquad (\inf\emptyset=\i).$$

A Young function $\Phi$ is said to satisfy the $\Delta_2$-condition, denoted by $\Phi \in \Delta_2$, if
$$
\Phi(2r)\le k\Phi(r), \qquad r>0
$$
for some $k>1$.
If $\Phi \in \Delta_2$, then $\Phi \in \mathcal{Y}$.
A Young function $\Phi$ is said to satisfy the $\nabla_2$-condition,
denoted also by $\Phi \in \nabla_2$, if
$$\Phi(r)\leq \frac{1}{2k}\Phi(kr),\qquad r\geq 0$$
for some $k>1$.
The function $\Phi(r) = r$ satisfies the $\Delta_2$-condition
and it fails the $\nabla_2$-condition.
If $1 < p < \i$,
then $\Phi(r) = r^p$ satisfies both the conditions.
The function $\Phi(r) = e^r - r - 1$ satisfies the
$\nabla_2$-condition but it fails the $\Delta_2$-condition.

For a Young function $\Phi$,
the complementary function $\widetilde{\Phi}(r)$ is defined by
\begin{equation*}
\widetilde{\Phi}(r)
\equiv
\left\{
\begin{array}{ccc}
\sup\{rs-\Phi(s): s\in [0,\i)\}
& \mbox{ if } & r\in [0,\i), \\
\i&\mbox{ if }& r=\i.
\end{array}
\right.
\end{equation*}
The complementary function $\widetilde{\Phi}$
is also a Young function and
it satisfies $\widetilde{\widetilde{\Phi}}=\Phi$.
Note that $\Phi \in \nabla_2$ if and only if $\widetilde{\Phi} \in \Delta_2$.

It is also known that
\begin{equation}\label{2.3}
r\leq \Phi^{-1}(r)\widetilde{\Phi}^{-1}(r)\leq 2r, \qquad r\geq 0.
\end{equation}

We recall an important pair of indices used for Young functions. For any Young function $\Phi$, write
$$
h_{\Phi}(t)=\sup_{s>0}\frac{\Phi(st)}{\Phi(s)}, \quad t>0.
$$
The lower and upper dilation indices of $\Phi$ are defined by
$$
i_{\Phi}=\lim_{t\to 0^+}\frac{\log h_{\Phi}(t)}{\log t} \text{ and } I_{\Phi}=\lim_{t\to \infty}\frac{\log h_{\Phi}(t)}{\log t},
$$
respectively.

A Young function $\Phi$ is said to be of upper type $p$ (resp. lower type $p$) for some $p\in[0,\i)$, if there exists a positive constant $C$ such that, for all $t\in[1,\i)$ (resp. $t\in[0,1]$) and $s\in[0,\i)$,
\begin{equation}\label{uplowtyp}
\Phi(st)\le Ct^p\Phi(s).
\end{equation}

\begin{rem}\label{remlowup}
It is well known that if $\Phi$ is of lower type $p_0$ and upper type $p_1$ with $1<p_0\le p_1<\i$,
then $\widetilde{\Phi}$ is of lower type $p^{\prime}_1$ and upper type $p^{\prime}_0$
and $\Phi$ is lower type $p_0$ and upper type $p_1$ with $1<p_0\le p_1<\i$ if and only if $\Phi\in \Delta_2\cap\nabla_2$.
\end{rem}

It is easy to see that $\Phi$ is of lower type $i_{\Phi}-\varepsilon$, and of upper type $I_{\Phi}+\varepsilon$ for every $\varepsilon > 0$, where the constant appearing in \eqref{uplowtyp} may depend on $\varepsilon$. We also mention that $i_{\Phi}$ and $I_{\Phi}$ may be viewed as the supremum of the lower types of $\Phi$ and the infimum of upper types, respectively.


\begin{defn}\label{ttss}
For a Young function $\Phi$
and $w\in A_{\infty}$,
the set
$$L^{\Phi}_{w}(\Rn)
\equiv
\left\{f-\text{measurable} :
\int_{\Rn}\Phi(k|f(x)|) w(x)dx<\i \text{ for some $k>0$ }\right\}
$$
is called the weighted Orlicz space.
The local weighted Orlicz space $L^{\Phi,\rm loc}_{w}(\Rn)$
is defined as the set of all functions $f$
such that $f\chi_{_B}\in L^{\Phi}_{w}(\Rn)$ for all balls $B \subset \Rn$.
\end{defn}
Note that $L^{\Phi}_{w}(\Rn)$ is a Banach space
with respect to the norm
$$
\|f\|_{L^{\Phi}_{w}(\Rn)}\equiv\|f\|_{L^{\Phi}_{w}}=
\inf\left\{\lambda>0:
\int_{\Rn}\Phi\Big(\frac{|f(x)|}{\lambda}\Big)w(x)dx\leq 1\right\}
$$
and
\begin{equation*}
  \int_{\Rn}\Phi\Big(\frac{|f(x)|}{\|f\|_{L^{\Phi}_{w}}}\Big)w(x)dx\leq 1.
\end{equation*}

The following analogue of the H\"older inequality is known.
\begin{equation}\label{Hldw}
\left|\int_{\Rn}f(x)g(x)w(x)dx\right|\leq 2 \|f\|_{L^{\Phi}_{w}} \|g\|_{L^{\widetilde{\Phi}}_{w}}.
\end{equation}
For the proof of \eqref{2.3} and \eqref{Hldw}, see, for example \cite{RaoRen}.


For a weight $w$, a measurable function $f$ and $t>0$, let
$$
m(w,\ f,\ t)=w(\{x\in\Rn:|f(x)|>t\}).
$$

\begin{defn}\label{wwos} The weak weighted Orlicz space
$$
WL^{\Phi}_w(\mathbb{R}^{n})=\{f-\text{measurable}:\Vert f\Vert_{WL^{\Phi}_w}<\infty\}
$$
is defined by the norm
$$
\|f\|_{WL^{\Phi}_{w}(\Rn)}\equiv\|f\|_{WL^{\Phi}_{w}}=\inf\Big\{\lambda>0\ :\ \sup_{t>0}\Phi(t)m\Big(w,\,\frac{f}{\lambda},\ t\Big)\ \leq 1\Big\}.
$$
\end{defn}

We can prove the following by a direct calculation:
\begin{equation}\label{charorlw}
\|\chi_{_B}\|_{L^{\Phi}_{w}}=\|\chi_{_B}\|_{WL^{\Phi}_{w}}
= \frac{1}{\Phi^{-1}\left(w(B)^{-1}\right)},\quad B\in\mathcal{B},
\end{equation}
where $\chi_{_B}$ denotes the characteristic function of the $B$.

The Hardy-Littlewood maximal operator $M$ is defined by
$$
Mf(x)=\sup_{r>0}\frac{1}{|B(x,r)|}\int_{B(x,r)}|f(y)|dy,\qquad x\in\Rn
$$
for a locally integrable function $f$ on $\Rn$.

The Calder\'{o}n-Zygmund (singular integral) operator $T$ is a bounded
linear operator on $L^2(\Rn)$ for which there exists a function $K$ on $\Rn\times\Rn$ that satisfies
the following conditions:
\begin{itemize}
\item[(i)] There exists $C>0$ such that $|K(x,y)|\leq\frac{C}{|x-y|^{n}}$ for $x\neq y.$
\item[(ii)] There exists $\varepsilon>0$ and $C>0$ such that
\begin{align*}
|K(x,y)-K(z,y)|+|K(y,x)-K(y,z)|\leq C\frac{|x-z|^{\varepsilon}}{|x-y|^{n+\varepsilon}}
\end{align*}
whenever $|x-y|\geq2|x-z|$ with $x\neq y.$
\item[(iii)] If $f\in L^{\infty}_{\rm comp}(\Rn)$, then
\begin{equation} \label{SIO1}
Tf(x) =  \int_{\Rn} K(x,y) f(y) \,dy
\end{equation}
for all $x\in\Rn\setminus \text{supp}(f).$
\end{itemize}

\begin{rem} \label{GURB01}
One can prove that $T$ is of weak type $(1, 1)$ and type $(p, p)$, $1 < p < \i$, for $f\in L^{\infty}_{\rm comp}(\Rn)$, and then $T$ is uniquely extended to an $L^p$-bounded operator by the density of $L^{\infty}_{\rm comp}(\Rn)$ in $L^p(\Rn)$; for example, see \cite{Stein93}. On the other hand, $L^{\infty}_{\rm comp}(\Rn)$ is not dense in Morrey spaces in general. Therefore, we need to give a precise definition of $Tf$ for the function $f$ in Morrey spaces, for example,
\begin{equation*}
T f(x)=  T \big( f \chi_{_{2B}}\big) + \int_{\Rn \setminus (2B)} K(x,y) f(y)  dy,
\end{equation*}
for some ball $B$ which contains $x$, with proving the absolute convergence of the integral in the second term and the independence of the choice of the
ball $B$ (cf. \cite{KwokJapSoc,Nakai2}).
\end{rem}


\

\section{Generalized weighted Orlicz-Morrey spaces}

In this section, we give the definition of the generalized weighted Orlicz-Morrey spaces $M^{\Phi,\varphi}_{w}({\mathbb R}^n)$ and investigate the fundamental structure of $M^{\Phi,\varphi}_{w}({\mathbb R}^n)$. In the sequel we use the notation $\|f\|_{L^{\Phi}_{w}(B)}\equiv\|f\chi_{_B}\|_{L^{\Phi}_{w}}$, $\varphi(B)\equiv\varphi(x,r)$ and $cB\equiv B(x,cr)$ for $B=B(x,r)\in \mathcal{B}$ and $c>0$.
\begin{defn}\label{weigenOrlMor}
Let $\varphi$ be a positive measurable function
on $\Rn\times (0,\i)$, let $w$ be a non-negative measurable function on $\Rn$
and $\Phi$ any Young function.
Denote by $M^{\Phi,\varphi}_{w}({\mathbb R}^n)$
the generalized weighted Orlicz-Morrey space,
the space of all functions $f\in L^{\Phi,\rm loc}_{w}({\mathbb R}^n)$
such that
\begin{align*}
\|f\|_{M^{\Phi,\varphi}_{w}(\Rn)}\equiv\|f\|_{M^{\Phi,\varphi}_{w}} &=
\sup\limits_{x\in\Rn, r>0} \varphi(x,r)^{-1} \,
\Phi^{-1}\big(w(B(x,r))^{-1}\big) \, \|f\|_{L^{\Phi}_{w}(B(x,r))}\\
&\equiv \sup\limits_{B\in\mathcal{B}} \varphi(B)^{-1} \,
\Phi^{-1}\big(w(B)^{-1}\big) \, \|f\|_{L^{\Phi}_{w}(B)}<\i.
\end{align*}
\end{defn}
\begin{ex}
Let $1\le p<\infty$ and $0<\kappa<1$.
\begin{itemize}
\item
If $\Phi(r)=r^p$ and $\varphi(x,r)=w(B(x,r))^{-1/p}$, then $M^{\Phi,\varphi}_{w}({\mathbb R}^n)=L^{p}_{w}({\mathbb R}^n)$.
\item
If $\Phi(r)=r^p$ and $\varphi(x,r)=w(B(x,r))^{\frac{\kappa-1}{p}}$, then $M^{\Phi,\varphi}_{w}({\mathbb R}^n)=L^{p,\kappa}(w)$.
\item
If $\Phi(r)=r^p$, then $M^{\Phi,\varphi}_{w}({\mathbb R}^n)=M^{p,\varphi}_{w}({\mathbb R}^n)$.
\item
If $\varphi(x,r)=\Phi^{-1}\big(w(B(x,r))^{-1}\big)$, then $M^{\Phi,\varphi}_{w}({\mathbb R}^n)=L^{\Phi}_{w}({\mathbb R}^n)$.
\end{itemize}
\end{ex}
For a Young function $\Phi$ and a non-negative measurable function $w$, we denote by ${\mathcal{G}}_{\Phi}^{w}$ the set of all functions $\varphi:\Rn\times(0,\infty) \to (0,\infty)$
such that
$$
\inf\limits_{B\in\mathcal{B}; \, r_{B} \leq r_{B_0}}\varphi(B) \gtrsim \varphi(B_0) \quad \text{for all}~ B_0\in\mathcal{B}
$$
and
$$
\inf\limits_{B\in\mathcal{B}; \, r_{B}\geq r_{B_0}}\frac{\varphi(B)}{\Phi^{-1}\big(w(B)^{-1}\big)} \gtrsim \frac{\varphi(B_0)}{\Phi^{-1}\big(w(B_0)^{-1}\big)}  \quad \text{for all}~ B_0\in\mathcal{B},
$$
where $r_{B}$ and $r_{B_0}$ denote the radius of the balls $B$ and $B_0$, respectively.

\begin{lem}\label{charwOrlMor}
Let $B_0:=B(x_0,r_0)$. If $\varphi\in{\mathcal{G}}_{\Phi}^{w}$, then there exists $C>0$ such that
$$
\frac{1}{\varphi(B_0)}\leq \|\chi_{B_0}\|_{M^{\Phi,\varphi}_{w}}\leq \frac{C}{\varphi(B_0)}.
$$
\end{lem}
\begin{proof}
Let $B=B(x,r)$ denote an arbitrary ball in $\Rn$. By the definition and \eqref{charorlw}, it is easy to see that
\begin{align*}
\|\chi_{B_0}\|_{M^{\Phi,\varphi}_{w}}&=\sup\limits_{B\in\mathcal{B}}\varphi(B)^{-1}\Phi^{-1}(w(B)^{-1})\frac{1}{\Phi^{-1}(w(B\cap B_0)^{-1})}\\
&\geq \varphi(B_0)^{-1}\Phi^{-1}(w(B_0)^{-1})\frac{1}{\Phi^{-1}(w(B_0\cap B_0)^{-1})}=\frac{1}{\varphi(B_0)}.
\end{align*}

Now if $r\leq r_0$, then $\varphi(B_0)\leq C\varphi(B)$ and
\begin{align*}
\varphi(B)^{-1}\Phi^{-1}(w(B)^{-1})\|\chi_{B_0}\|_{L^{\Phi}_{w}(B)}\leq\frac{1}{\varphi(B)}\leq \frac{C}{\varphi(B_0)}.
\end{align*}

On the other hand if $r\geq r_0$, then $\frac{\varphi(B_0)}{\Phi^{-1}(w(B_0)^{-1})}\leq C\frac{\varphi(B)}{\Phi^{-1}(w(B)^{-1})}$ and
\begin{align*}
\varphi(B)^{-1}\Phi^{-1}(w(B)^{-1})\|\chi_{B_0}\|_{L^{\Phi}_{w}(B)}\leq \frac{C}{\varphi(B_0)}.
\end{align*}
This completes the proof.
\end{proof}

We denote by $WM^{\Phi,\varphi}_{w}({\mathbb R}^n)$
the weak generalized weighted Orlicz-Morrey space,
the space of all functions $f\in WL^{\Phi,\rm loc}_{w}({\mathbb R}^n)$
such that
$$
\|f\|_{WM^{\Phi,\varphi}_{w}}=
\sup\limits_{x\in\Rn, r>0} \varphi(x,r)^{-1} \,
\Phi^{-1}\big(w(B(x,r))^{-1}\big) \, \|f\|_{WL^{\Phi}_{w}(B(x,r))}<\i.
$$

\begin{lem}\label{charwOrlMorW}
Let $B_0:=B(x_0,r_0)$. If $\varphi\in{\mathcal{G}}_{\Phi}^{w}$, then there exists $C>0$ such that
$$
\frac{1}{\varphi(B_0)}\leq \|\chi_{B_0}\|_{WM^{\Phi,\varphi}_{w}}\leq \frac{C}{\varphi(B_0)}.
$$
\end{lem}
\begin{proof}
The proof could be made similarly to the proof of Lemma \ref{charwOrlMor} thanks to \eqref{charorlw}.
\end{proof}

\

\section{Calder\'{o}n-Zygmund operators}

$~~~$ In this section necessary and sufficient conditions for the weak/strong boundedness of the Calder\'{o}n-Zygmund operator $T$ in generalized weighted Orlicz-Morrey spaces will be obtained.

Before the presentation of the main results, we recall some crucial inequalities to establish the boundedness of
singular integral operator in generalized weighted Orlicz-Morrey spaces.

\begin{thm}\label{MaxweigOrlcWMAx}\cite[Proposition 2.4]{GogKrb1994}
Let $\Phi$ be a Young function.
Assume in addition $w \in A_{i_{\Phi}}$. Then, there is a constant $C>1$ such that
\begin{equation*}
\Phi(t)m\Big(w,\,Mf,\ t\Big)\leq C \int_{\Rn}\Phi\left(C|f(x)|\right)w(x)dx
\end{equation*}
for every locally integrable $f$ and every $t>0$.
\end{thm}

\begin{rem}\label{remwmi}
For a sublinear operator $S$, weak modular inequality
\begin{equation}\label{weakmodinqsub}
\Phi(t)m\Big(w,\,Sf,\ t\Big)\leq C \int_{\Rn}\Phi\left(C|f(x)|\right)w(x)dx
\end{equation}
implies the corresponding norm inequality.
Indeed, let \eqref{weakmodinqsub} holds. Then, we have
\begin{align*}
\Phi(t)w\left(\{x\in\Rn:\frac{|Sf(x)|}{C^2\Vert f\Vert_{L^{\Phi}_w}}>t\}\right)&=\Phi(t)w\left(\{x\in\Rn:\left|S\Big(\frac{f}{C^2\Vert f\Vert_{L^{\Phi}_w}}\Big)(x)\right|>t\}\right)
\\
&\leq C \int_{\Rn}\Phi\left(\frac{|f(x)|}{C\Vert f\Vert_{L^{\Phi}_w}}\right)w(x)dx\leq 1,
\end{align*}
which implies $\|Sf\|_{WL^{\Phi}_w}\lesssim \|f\|_{L^{\Phi}_w}$.
\end{rem}

\begin{thm}\cite[Theorem 2.2]{GogKrb1995}\label{MaxweigOrlcW}
Let $\Phi$ be a Young function with $\Phi\in\Delta_2$.
Assume in addition $w \in A_{i_{\Phi}}$. Then, there is a constant $C>1$ such that
\begin{equation*}
\Phi(t)m\Big(w,\,Tf,\ t\Big)\leq C \int_{\Rn}\Phi\left(C|f(x)|\right)w(x)dx
\end{equation*}
for every locally integrable $f$ and every $t>0$.
\end{thm}

\begin{thm}\label{SIOweigOrlcold}\cite[Theorem 3.2]{GogKrb1995}
Let $\Phi$ be a Young function with $\Phi\in\Delta_2\cap\nabla_2$.
Assume in addition $w \in A_{i_{\Phi}}$. Then, there is a constant $C\geq1$ such that
\begin{equation}\label{strmodinqSIO}
\int_{\Rn}\Phi\left(|Tf(x)|\right)w(x)dx\leq C \int_{\Rn}\Phi\left(|f(x)|\right)w(x)dx
\end{equation}
for any locally integrable function $f$.
\end{thm}

\begin{rem}
It is well known that the modular inequality \eqref{strmodinqSIO} implies the boundedness on $L^{\Phi}_{w}(\Rn)$. However, the converse is not necessarily true.
\end{rem}

\begin{rem}
Nakai \cite{Nakai2} proved the unweighted version of modular inequality \eqref{strmodinqSIO}. Cianchi \cite{Cianchi} gave a necessary and sufficient condition of $\Phi$ and $\Psi$ for the boundedness from an Orlicz space $L^{\Phi}(\Rn)$ to another Orlicz space $L^{\Psi}(\Rn)$. He treated singular integral operators $T$ with the form
$$
Tf(x)=\lim_{\varepsilon\to 0^{+}}\int_{|y|>\varepsilon}\frac{\Omega(y)}{|y|^n}f(x-y)dy,
$$
where $\Omega$ is an odd function on $\Rn$ which is homogeneous of degree $0$ and satisfies the 'Dini-type' condition. These operators are Calder\'{o}n-Zygmund operators with $K(x,y)=\frac{\Omega(x-y)}{|x-y|^n}$.

Note that the modular inequality \eqref{strmodinqSIO} was also investigated by Poelhuis and Torchinsky in \cite[Theorem 5.3]{PoelTor}.
\end{rem}

\begin{lem}\label{lemHold}
	Let $\Phi$ be a Young function and $f \in L^{\Phi}_{w, {\rm loc}}(\Rn)$.
	Assume in addition $w \in A_{i_{\Phi}}$.
	For a ball $B$,
	the following inequality is valid:
	$$
	\|f\|_{L^{1}(B)} \lesssim |B|
	\Phi^{-1}\left(w(B)^{-1}\right) \|f\|_{L^{\Phi}_{w}(B)}.
	$$
\end{lem}
\begin{proof}
	Let
	$$\mathfrak{M}f(x)=\sup_{B\in\mathcal{B}}\frac{\chi_{_B}(x)}{|B|}\int_{B}|f(y)|dy,\quad x\in \Rn$$
	and $\tilde{f}$ denotes the extension of $f$ from $B$ to $\Rn$ by zero. It is well known that $\mathfrak{M}f(x)\leq 2^n Mf(x)$ for all $x\in\Rn$. Then using Theorem \ref{MaxweigOrlcWMAx}, we have
	\begin{align*}
	&\frac{\|f\|_{L^1(B)}}{|B|}
	\|\chi_B\|_{WL^\Phi_w(B)}
	=\frac{\|\tilde{f}\|_{L^1(B)}}{|B|}
	\|\chi_B\|_{WL^\Phi_w(B)} \lesssim
	\|\mathfrak{M}\tilde{f}\|_{WL^\Phi_w(B)}
	\\
	&\lesssim
	\|M\tilde{f}\|_{WL^\Phi_w(B)}
	\le
	\|M\tilde{f}\|_{WL^\Phi_w(\Rn)}
	\lesssim
	\|\tilde{f}\|_{L^{\Phi}_{w}(\Rn)}=\|f\|_{L^{\Phi}_{w}(B)}.
	\end{align*}
	So, Lemma \ref{lemHold} is proved.
\end{proof}

The following boundedness result for the Hardy-Littlewood maximal operator on generalized weighted Orlicz-Morrey spaces is valid.
\begin{thm}\label{thm4.4.max}\cite{DerGulHasTJM}
Let $\Phi\in\nabla_2$ and $w \in A_{i_{\Phi}}$. If $\varphi_1\in{\mathcal{G}}^{\Phi}_w$, then the condition
\begin{equation}\label{condMnec}
\varphi_1(x,r)\le C \varphi_2(x,r),
\end{equation}
where $C$ does not depend on $x$ and $r$,
is necessary and sufficient for the boundedness of $M$ from $M^{\Phi,\varphi_1}_{w}({\mathbb R}^n)$ to $M^{\Phi,\varphi_2}_{w}({\mathbb R}^n)$.
\end{thm}
If we take $\varphi_1=\varphi_2=\varphi$ at Theorem \ref{thm4.4.max} we get the following corollary.
\begin{cor}\label{thm4.4.maxcor}
If $\Phi\in\nabla_2$, $w \in A_{i_{\Phi}}$ and $\varphi\in{\mathcal{G}}_{\Phi}^{w}$, then the Hardy-Littlewood maximal operator $M$ is bounded on $M^{\Phi,\varphi}_{w}({\mathbb R}^n)$.
\end{cor}

We recall that $M$ is weak generalized weighted Orlicz-Morrey bounded as the following theorem shows:
\begin{thm}\label{thm4.4.maxw}\cite{DerGulHasTJM}
Let $\Phi$ be a Young function and $w \in A_{i_{\Phi}}$. If $\varphi_1\in{\mathcal{G}}^{\Phi}_w$, then the condition \eqref{condMnec}
is necessary and sufficient for the boundedness of $M$ from $M^{\Phi,\varphi_1}_{w}({\mathbb R}^n)$ to $WM^{\Phi,\varphi_2}_{w}({\mathbb R}^n)$.
\end{thm}

%

Sufficient conditions on $(\Phi,\varphi_1,\varphi_2,w)$ for the boundedness of the Calder\'{o}n-Zygmund operators from one generalized weighted Orlicz-Morrey space $M^{\Phi,\varphi_1}_{w}$ to another $M^{\Phi,\varphi_2}_{w}$ as stated in the following theorem.
\begin{thm}\label{sufsio}
Let $\varphi_1,\varphi_2$ be positive measurable functions on $\Rn\times (0,\i)$ with satisfying the condition
\begin{equation}\label{es1}
\varphi_1(x,2r)\leq C\varphi_2(x,r),
\end{equation}
where $C$ does not depend on $x\in\Rn$ and $r>0$. Let also $\Phi$ be a Young function with $\Phi\in\Delta_2\cap\nabla_2$ and $w \in A_{i_{\Phi}}$. Then the condition
\begin{equation}\label{wgtcond}
\int_{r}^{\infty}\varphi_1(x,t)\frac{dt}{t}\le C \varphi_2(x,r),
\end{equation}
where $C$ does not depend on $x$ and $r$, is sufficient for the boundedness of Calder\'{o}n-Zygmund operator $T$ from $M^{\Phi,\varphi_1}_{w}({\mathbb R}^n)$ to $M^{\Phi,\varphi_2}_{w}({\mathbb R}^n)$.
\end{thm}

\begin{proof}
Taking into account Remark \ref{GURB01} we proceed as in the following.

For the moment, we denote the singular integral operator on $L^{\Phi}_{w}(\Rn)$ by $T_{0}$ to avoid confusion.
For $f\in M^{\Phi,\varphi_1}_{w}({\mathbb R}^n)$ and $x\in\Rn$ we choose a ball $B=B(x_0,r)\in \mathcal{B}$ such that $x\in B$, and let
\begin{equation}\label{defsinmor}
Tf(x):=T_{0}f_{1}(x)+\int_{\Rn}K(x,y)f_{2}(y)dy,\qquad f=f_1+f_2,~~f_1=f\chi _{2B}.
\end{equation}

First we show that $Tf(x)$ is well-defined $a.e.$ $x$ and independent of the choice $B$ containing $x$.

As $T_{0}$ is bounded on $L^{\Phi}_{w}(\Rn)$ provided by Theorem \ref{SIOweigOrlcold} and $f_1\in L^{\Phi}_{w}(\Rn)$, $T_{0}f_{1}$ is well-defined.

Next, we show that the second-term of the right-hand side defining $Tf(x)$ converges absolutely for any $f\in M^{\Phi,\varphi_1}_{w}({\mathbb R}^n)$ and almost every $x\in\Rn$.

Observe  that the inclusions   $x\in B$, $y\in \Rn\setminus 2B=\dual{(2B)}$ imply
$\frac{1}{2}|x_0-y|\le |x-y|\le\frac{3}{2}|x_0-y|$. Then we get
$$
\int_{\Rn}|K(x,y)f_{2}(y)|dy \lesssim \int_{\dual {(2B)}}\frac{|f(y)|}{|x-y|^{n}}dy \lesssim \int_{\dual {(2B)}}\frac{|f(y)|}{|x_0-y|^{n}}dy.
$$

By Fubini's theorem we have
\begin{equation*}
\begin{split}
\int_{\dual {(2B)}}\frac{|f(y)|}{|x_0-y|^{n}}dy &
\thickapprox
\int_{\dual {(2B)}}|f(y)|\int_{|x_0-y|}^{\i}\frac{dt}{t^{n+1}}dy
\\
&\thickapprox \int_{2r}^{\i}\int_{2r\leq |x_0-y|< t}|f(y)|dy\frac{dt}{t^{n+1}}
\\
&\lesssim \int_{2r}^{\i}\int_{B(x_0,t) }|f(y)|dy\frac{dt}{t^{n+1}}.
\end{split}
\end{equation*}
Applying Lemma \ref{lemHold}, we get
\begin{equation}\label{sal00}
\int_{\dual {(2B)}}\frac{|f(y)|}{|x_0-y|^{n}}dy \lesssim
\int_{2r}^{\i}\|f\|_{L^{\Phi}_{w}(B(x_0,t))} \Phi^{-1}\big(w(B(x_0,t))^{-1}\big) \frac{dt}{t}.
\end{equation}
Lastly, by condition \eqref{wgtcond} we obtain for all $x\in B$
\begin{equation}\label{sal00xyz}
\int_{\Rn}|K(x,y)f_{2}(y)dy| \lesssim \|f\|_{M^{\Phi,\varphi_1}_{w}}\int_{r}^{\i}\varphi_1(x_0,t)\frac{dt}{t}\lesssim \|f\|_{M^{\Phi,\varphi_1}_{w}}\,\varphi_2(x_0,r)<\infty.
\end{equation}

Finally it remains to show that the definition is independent of the choice of $B$. That is, if $B_{1},{B_2}\in \mathcal{B}$ and $x\in B_{1}\cap{B_2}$, then
\begin{equation}\label{indchB}
T_{0}(f\chi _{2B_1})(x)+\int_{\Rn\setminus 2B_1}K(x,y)f(y)dy=T_{0}(f\chi _{2B_2})(x)+\int_{\Rn\setminus 2B_2}K(x,y)f(y)dy.
\end{equation}
Actually, let $B_3\in\mathcal{B}$ be selected so that $2B_{1}\cup{2B_2}\subset B_3$.
Since $f\chi _{2B_1}, f\chi _{B_3\setminus 2B_1} \in L^{\Phi}_{w}(\Rn)$, the linearity of $T_0$ on $L^{\Phi}_{w}(\Rn)$ yields
\begin{align}\label{indbchs1}
  &T_{0}(f\chi _{2B_1})(x)+\int_{\Rn\setminus 2B_1}K(x,y)f(y)dy \notag \\
  &= T_{0}(f\chi _{2B_1})(x) + \int_{B_3\setminus 2B_1}K(x,y)f(y)dy+ \int_{\Rn\setminus B_3}K(x,y)f(y)dy \notag \\
  &= T_{0}(f\chi _{2B_1})(x) + T_{0}(f\chi _{B_3\setminus 2B_1})(x)+ \int_{\Rn\setminus B_3}K(x,y)f(y)dy \notag \\
  &= T_{0}(f\chi _{B_3})(x) + \int_{\Rn\setminus B_3}K(x,y)f(y)dy.
\end{align}

Similarly, we also have
\begin{align}\label{indbchs2}
  T_{0}(f\chi _{2B_2})(x)+\int_{\Rn\setminus 2B_2}K(x,y)f(y)dy = T_{0}(f\chi _{B_3})(x) + \int_{\Rn\setminus B_3}K(x,y)f(y)dy.
\end{align}
Therefore, combining \eqref{indbchs1} and \eqref{indbchs2} we obtain \eqref{indchB}.

Now, we show the boundedness.

Since $f_1\in L^{\Phi}_{w}(\Rn)$, by the boundedness of $T_{0}$ in $L^{\Phi}_{w}(\Rn)$ provided by Theorem \ref{SIOweigOrlcold}, it follows that
\begin{equation}\label{es2}
\|T_{0}f_1\|_{L^{\Phi}_{w}(B)}\leq \|T_{0}f_1\|_{L^{\Phi}_{w}(\Rn)}\lesssim
\|f_1\|_{L^{\Phi}_{w}(\Rn)}=\|f\|_{L^{\Phi}_{w}(2B)}.
\end{equation}

We also have
\begin{equation}\label{es3}
\frac{1}{\Phi^{-1}\big(w(2B)^{-1}\big)}=\|\chi_{2B}\|_{WL^{\Phi}_{w}}
\lesssim
\|M\chi_{B}\|_{WL^{\Phi}_{w}}
\lesssim
\|\chi_{B}\|_{L^{\Phi}_{w}}=\frac{1}{\Phi^{-1}\big(w(B)^{-1}\big)}
\end{equation}
from the well-known pointwise estimate
$\chi_{2B}(z) \lesssim M\chi_{B}(z), \text{    for all  } z \in {\mathbb R}^n$
and Theorem \ref{MaxweigOrlcWMAx}.

By combining \eqref{es1}, \eqref{es2} and \eqref{es3}, we get the estimate
\begin{align}\label{es4ftg}
\varphi_2(B)^{-1} \,
\Phi^{-1}\big(w(B)^{-1}\big) \, \|T_{0}f_1\|_{L^{\Phi}_{w}(B)}
& \lesssim \varphi_1(2B)^{-1} \,
\Phi^{-1}\big(w(2B)^{-1}\big) \, \|f\|_{L^{\Phi}_{w}(2B)}
\notag
\\
& \lesssim\|f\|_{M^{\Phi,\varphi_1}_{w}}.
\end{align}

From \eqref{sal00xyz} for all $x\in B$, we have
\begin{align}\label{gfvjzdhhs}
  |Tf(x)| &\leq |T_{0}f_{1}(x)|+\int_{\Rn}|K(x,y)f_{2}(y)dy| \notag \\
  & \lesssim |T_{0}f_{1}(x)| + \|f\|_{M^{\Phi,\varphi_1}_{w}}\,\varphi_2(B).
\end{align}
Applying the norm $\|\cdot\|_{L^{\Phi}_{w}}$ on both sides of \eqref{gfvjzdhhs}, then by \eqref{charorlw} we get,
\begin{equation*}
\|Tf\|_{L^{\Phi}_{w}(B)}\lesssim \|T_{0}f_1\|_{L^{\Phi}_{w}(B)} + \frac{\|f\|_{M^{\Phi,\varphi_1}_{w}}\,\varphi_2(B)}{\Phi^{-1}\big(w(B)^{-1}\big)}.
\end{equation*}
Consequently, by \eqref{es4ftg} we have
\begin{align*}
&\varphi_2(B)^{-1} \,
\Phi^{-1}\big(w(B)^{-1}\big)\|Tf\|_{L^{\Phi}_{w}(B)} \notag \\
&~~~~\lesssim \varphi_2(B)^{-1} \,
\Phi^{-1}\big(w(B)^{-1}\big)\|T_{0}f_1\|_{L^{\Phi}_{w}(B)}+ \|f\|_{M^{\Phi,\varphi_1}_{w}}  \lesssim \|f\|_{M^{\Phi,\varphi_1}_{w}}.
\end{align*}
By taking supremum over $B\in\mathcal{B}$, we obtain the boundedness of $T$ from $M^{\Phi,\varphi_1}_{w}({\mathbb R}^n)$ to $M^{\Phi,\varphi_2}_{w}({\mathbb R}^n)$.
\end{proof}

%
%
%

If we take $\varphi_1=\varphi_2=\varphi$ at Theorem \ref{sufsio} we get the following corollary.
\begin{cor}\label{sufsiocor}
Let $\varphi$ be a positive measurable function
on $\Rn\times (0,\i)$ with satisfying doubling condition uniformly over the first variable, that is, there exist constant $C>0$
$$
\varphi(x,r)\geq C\varphi(x,2r),
$$
for all $x\in\Rn$ and $r>0$. Let also $\Phi$ be a Young function with $\Phi\in\Delta_2\cap\nabla_2$ and $w \in A_{i_{\Phi}}$. Then the condition
\begin{equation}\label{patcas}
\int_{r}^{\infty}\varphi(x,t)\frac{dt}{t}\le C \varphi(x,r),
\end{equation}
where $C$ does not depend on $x$ and $r$, is sufficient for the boundedness of Calder\'{o}n-Zygmund operator $T$ on $M^{\Phi,\varphi}_{w}({\mathbb R}^n)$.
\end{cor}

If we take $\Phi(t)=t^p,~1\le p<\infty$ at Theorem \ref{sufsio} we get the following corollary for generalized weighted Morrey spaces which was proved in \cite{KarGulSer}.
\begin{cor}
Let $1< p<\infty$, $w \in A_{p}$ and $\varphi_1,\varphi_2$ be positive measurable functions on $\Rn\times (0,\i)$ with satisfying the condition
\eqref{es1}. Then the condition \eqref{wgtcond} is sufficient for the boundedness of Calder\'{o}n-Zygmund operator $T$ from $M^{p,\varphi_1}_{w}({\mathbb R}^n)$ to $M^{p,\varphi_2}_{w}({\mathbb R}^n)$.
\end{cor}

To obtain a necessary condition for the boundedness of the singular integral $T$ on $M^{\Phi,\varphi}_{w}({\mathbb R}^n)$ we shall consider the genuine singular integral operator, defined by Burenkov et al. \cite{BurGulSerTar}.

\begin{defn}A Calder\'{o}n-Zygmund operator $T$ is called a genuine singular integral
operator if there exist some constant $C > 0$ and a cone $V:=R\{x\in\Rn:|x'|<\theta|x_n|\}$ with $\theta>0$ and $R\in O(n)$ such that
\begin{align*}
K(x,y)\geq\frac{C}{|x-y|^{n}}
\end{align*}
for all $x,y\in\Rn$ with $x-y\in V.$ Here, $O(n)$ denotes the set of all orthogonal matrices
in $\Rn$.
\end{defn}

\begin{thm}\label{thm4.4.}
If $\Phi\in\nabla_2$, $w \in A_{i_{\Phi}}$, $\varphi_1,\varphi_2$ be positive measurable functions on $\Rn\times (0,\i)$ with satisfying the condition \eqref{condMnec} and $\varphi_1\in{\mathcal{G}}^{\Phi}_w$, then the condition \eqref{wgtcond} is necessary for the boundedness of genuine $T$ from $M^{\Phi,\varphi_1}_{w}({\mathbb R}^n)$ to $M^{\Phi,\varphi_2}_{w}({\mathbb R}^n)$.
\end{thm}

\begin{proof}
In the proof, we follow the ideas of \cite{HakNakSaw}. Assume that for any $m\in\mathbb{N} \cap [3,\i)$ there exists $B_m\in \mathcal{B}$ such that $\varphi_1(2^{m}B_m)^{-1}\leq2\varphi_2(B_m)^{-1}.$ Then, consider $f_m(x)=\chi_{_V}(-x)\chi_{2^{m-1}B_m\backslash 2B_m}(x)$. Let $x\in V\cap B_m$. If $y\in \text{supp}(f_m),$ then $x-y\in V$ and $r_{B_m}\leq |x-y|\leq2^{m}r_{B_m}$. Thus
\begin{align*}
Tf_m(x)&=\int_{\Rn}K(x,y)f_m(y)dy\\
&=\int_{\text{supp}(f_m)}K(x,y)dy\geq C\int_{\text{supp}(f_m)}\frac{dy}{|x-y|^n}\\
&\geq C\int_{\text{supp}(f_m)}\frac{dy}{|y|^n}=C\log m.
\end{align*}
Since V is a cone, we have $\chi_{B_m}\lesssim M\chi_{V\cap B_m}$. We use this estimate and the
boundedness of $M$ from $M^{\Phi,\varphi_1}_{w}({\mathbb R}^n)$ to $M^{\Phi,\varphi_2}_{w}({\mathbb R}^n)$, see Theorem \ref{thm4.4.max}, to obtain
\begin{align*}
\varphi_2(B_m)^{-1}\lesssim\|\chi_{B_m}\|_{M^{\Phi,\varphi_2}_{w}}\lesssim\|M\chi_{V\cap B_m}\|_{M^{\Phi,\varphi_2}_{w}}\lesssim\|\chi_{V\cap B_m}\|_{M^{\Phi,\varphi_1}_{w}}.
\end{align*}
By using the inequality $\log m\lesssim Tf_m(x)$ for $x\in V\cap B_m$ and the boundedness
of $T$ from $M^{\Phi,\varphi_1}_{w}({\mathbb R}^n)$ to $M^{\Phi,\varphi_2}_{w}({\mathbb R}^n)$, we have
\begin{align*}
\varphi_2(B_m)^{-1}\log m&\lesssim\|Tf_m\|_{M^{\Phi,\varphi_2}_{w}}\\
&\lesssim\|f_m\|_{M^{\Phi,\varphi_1}_{w}}\\
&\leq\|\chi_{2^{m}B_m}\|_{M^{\Phi,\varphi_1}_{w}}\lesssim\varphi_1(2^{m}B_m)^{-1}\lesssim\varphi_2(B_m)^{-1}.
\end{align*}
This implies $\log m\leq C$ where $C$ is independent of $m$, contradictory to the fact that
$m\geq3$ is arbitrary. Hence, there exists some $m_0\in\mathbb{N}$ such that $\varphi_1(2^{m_0} B ) < \frac{1}{2}\varphi_2(B)$ for all $B\in \mathcal{B}$.

Therefore,
\begin{align*}
\int_r^\i \varphi_1(x,t)\frac{dt}{t}&=\sum_{k=1}^\i\int_{2^{(k-1)m_0}r}^{2^{km_0}r}\varphi_1(B)\frac{dt}{t}
\\
&\leq\sum_{k=1}^\i\varphi_1(2^{(k-1)m_0}B)\int_{2^{(k-1)m_0}r}^{2^{km_0}r}\frac{dt}{t}
\\
&=m_0\log2\sum_{k=1}^\i\varphi_1(2^{(k-1)m_0}B)
\\
&\lesssim \varphi_2(B) \sum_{k=1}^\i\frac{1}{2^{k-1}}\lesssim \varphi_2(B).
\end{align*}
\end{proof}

If we take $\varphi_1=\varphi_2=\varphi$ at Theorem \ref{thm4.4.} we get the following corollary.
\begin{cor}
If $\Phi\in\nabla_2$, $\varphi \in{\mathcal{G}}^{\Phi}_w$ and $w \in A_{i_{\Phi}}$, then the condition \eqref{patcas} is necessary for the boundedness of genuine $T$ on $M^{\Phi,\varphi}_{w}({\mathbb R}^n)$.
\end{cor}

If we take $\Phi(t)=t^p,~1\le p<\infty$ at Theorem \ref{thm4.4.} we get the following corollary for generalized weighted Morrey spaces.
\begin{cor}
If $1<p<\infty$, $w \in A_{p}$, $\varphi_1,\varphi_2$ be positive measurable functions on $\Rn\times (0,\i)$ with satisfying the condition \eqref{condMnec} and $\varphi_1\in{\mathcal{G}}^{p}_w$, then the condition \eqref{wgtcond} is necessary for the boundedness of genuine Calder\'{o}n-Zygmund operator $T$ from $M^{p,\varphi_1}_{w}({\mathbb R}^n)$ to $M^{p,\varphi_2}_{w}({\mathbb R}^n)$.
\end{cor}

In particular, by combining Theorems \ref{sufsio} and \ref{thm4.4.} we have the following result.
\begin{cor}
Let $\Phi$ be a Young function with $\Phi\in\Delta_2\cap\nabla_2$ and $\varphi_1,\varphi_2$ be positive measurable functions on $\Rn\times (0,\i)$ with satisfying the condition \eqref{condMnec}. Let also $\varphi_1\in{\mathcal{G}}^{\Phi}_w$ and $w \in A_{i_{\Phi}}$. Then the condition \eqref{wgtcond} is necessary and sufficient for the boundedness of genuine Calder\'{o}n-Zygmund operator $T$ from $M^{\Phi,\varphi_1}_{w}({\mathbb R}^n)$ to $M^{\Phi,\varphi_2}_{w}({\mathbb R}^n)$.
\end{cor}

\begin{cor}
Let $1<p<\infty$ and $\varphi_1,\varphi_2$ be positive measurable functions on $\Rn\times (0,\i)$ with satisfying the condition \eqref{condMnec}. Let also $\varphi_1\in{\mathcal{G}}^{p}_w$ and  $w \in A_{p}$. Then the condition \eqref{wgtcond} is necessary and sufficient for the boundedness of genuine Calder\'{o}n-Zygmund operator $T$ from $M^{p,\varphi_1}_{w}({\mathbb R}^n)$ to $M^{p,\varphi_2}_{w}({\mathbb R}^n)$.
\end{cor}

\begin{rem}
It is obvious that if $\varphi_1\in{\mathcal{G}}^{\Phi}_w$ or $\varphi_1\in{\mathcal{G}}^{p}_w$, then condition \eqref{condMnec} implies the condition \eqref{es1}.
\end{rem}

The boundedness result for the singular integral operator on weak generalized weighted Orlicz-Morrey spaces is given in the following theorem.
\begin{thm}\label{sufsiow}
Let $\varphi_1,\varphi_2$ be positive measurable functions
on $\Rn\times (0,\i)$ with satisfying \eqref{es1}, $\Phi$ be a Young function with $\Phi\in\Delta_2$ and $w \in A_{i_{\Phi}}$. Then the condition \eqref{wgtcond} is sufficient
for the boundedness of Calder\'{o}n-Zygmund operator $T$ from $M^{\Phi,\varphi_1}_{w}({\mathbb R}^n)$ to $WM^{\Phi,\varphi_2}_{w}({\mathbb R}^n)$.
\end{thm}

\begin{proof}
Let $T$ be defined as in \eqref{defsinmor}. Since $f_1\in L^{\Phi}_{w}(\Rn)$, by the boundedness of $T$ from $L^{\Phi}_{w}(\Rn)$ to $WL^{\Phi}_{w}(\Rn)$ provided by Theorem \ref{MaxweigOrlcW}, it follows that
\begin{equation}\label{es2w}
\|T_{0}f_1\|_{WL^{\Phi}_{w}(B)}\leq \|T_{0}f_1\|_{WL^{\Phi}_{w}(\Rn)}\lesssim
\|f_1\|_{L^{\Phi}_{w}(\Rn)}=\|f\|_{L^{\Phi}_{w}(2B)}.
\end{equation}
By combining \eqref{es1}, \eqref{es2w} and \eqref{es3}, we get the estimate
\begin{align}\label{es4ftgw}
& \varphi_2(B)^{-1} \,
\Phi^{-1}\big(w(B)^{-1}\big) \, \|T_{0}f_1\|_{WL^{\Phi}_{w}(B)}  \notag
\\
&\lesssim \varphi_1(2B)^{-1} \,
\Phi^{-1}\big(w(2B)^{-1}\big) \, \|f\|_{L^{\Phi}_{w}(2B)}
\lesssim\|f\|_{M^{\Phi,\varphi_1}_{w}}.
\end{align}

Applying the norm $\|\cdot\|_{WL^{\Phi}_{w}}$ on both sides of \eqref{gfvjzdhhs}, then by \eqref{charorlw} we get,
\begin{equation*}
\|Tf\|_{WL^{\Phi}_{w}(B)}\lesssim \|T_{0}f_1\|_{WL^{\Phi}_{w}(B)} + \frac{\|f\|_{M^{\Phi,\varphi_1}_{w}}\,\varphi_2(B)}{\Phi^{-1}\big(w(B)^{-1}\big)}.
\end{equation*}
Consequently, from \eqref{es4ftgw} we have
\begin{align*}
&\varphi_2(B)^{-1} \,
\Phi^{-1}\big(w(B)^{-1}\big)\|Tf\|_{WL^{\Phi}_{w}(B)} \notag \\
&~~~~\lesssim \varphi_2(B)^{-1} \,
\Phi^{-1}\big(w(B)^{-1}\big)\|T_{0}f_1\|_{WL^{\Phi}_{w}(B)}+ \|f\|_{M^{\Phi,\varphi_1}_{w}}  \lesssim \|f\|_{M^{\Phi,\varphi_1}_{w}}.
\end{align*}
By taking supremum over $B\in\mathcal{B}$, we obtain the boundedness of $T$ from $M^{\Phi,\varphi_1}_{w}({\mathbb R}^n)$ to $WM^{\Phi,\varphi_2}_{w}({\mathbb R}^n)$.
\end{proof}

If we take $\Phi(t)=t^p,~1\le p<\infty$ at Theorem \ref{sufsiow} we get the following corollary for weak generalized weighted Morrey spaces  which was proved in \cite{KarGulSer}.
\begin{cor}
Let $1\le p<\infty$, $w \in A_{p}$ and $\varphi_1,\varphi_2$ be positive measurable functions on $\Rn\times (0,\i)$ with satisfying the condition
\eqref{es1}. Then the condition \eqref{wgtcond} is sufficient for the boundedness of Calder\'{o}n-Zygmund operator $T$ from $M^{p,\varphi_1}_{w}({\mathbb R}^n)$ to $WM^{p,\varphi_2}_{w}({\mathbb R}^n)$.
\end{cor}

By using Theorem \ref{thm4.4.maxw} and going through a similar argument to that in Theorem \ref{thm4.4.}, we also have a necessary condition for the boundedness of genuine singular integral operator on weak generalized weighted Orlicz-Morrey spaces, the details being omitted.

\begin{thm}\label{thm4.4.w}
If $\Phi$ be a Young function, $w \in A_{i_{\Phi}}$, $\varphi_1,\varphi_2$ be positive measurable functions on $\Rn\times (0,\i)$ with satisfying the condition \eqref{condMnec} and $\varphi_1\in{\mathcal{G}}^{\Phi}_w$, then the condition \eqref{wgtcond} is necessary for the boundedness of genuine $T$ from $M^{\Phi,\varphi_1}_{w}({\mathbb R}^n)$ to $WM^{\Phi,\varphi_2}_{w}({\mathbb R}^n)$.
\end{thm}

If we take $\Phi(t)=t^p,~1\le p<\infty$ at Theorem \ref{thm4.4.w} we get the following corollary for weak generalized weighted Morrey spaces.
\begin{cor}
If $1\le p<\infty$, $w \in A_{p}$, $\varphi_1,\varphi_2$ be positive measurable functions on $\Rn\times (0,\i)$ with satisfying the condition \eqref{condMnec} and $\varphi_1\in{\mathcal{G}}^{p}_w$, then the condition \eqref{wgtcond} is necessary for the boundedness of genuine $T$ from $M^{p,\varphi_1}_{w}({\mathbb R}^n)$ to $WM^{p,\varphi_2}_{w}({\mathbb R}^n)$.
\end{cor}

In particular, by combining Theorems \ref{sufsiow} and \ref{thm4.4.w} we have the following result.
\begin{cor}
Let $\Phi$ be a Young function with $\Phi\in\Delta_2$ and $\varphi_1,\varphi_2$ be positive measurable functions on $\Rn\times (0,\i)$ with satisfying the condition \eqref{condMnec}. Let also $\varphi_1\in{\mathcal{G}}^{\Phi}_w$ and $w \in A_{i_{\Phi}}$. Then the condition \eqref{wgtcond} is necessary and sufficient for the boundedness of genuine Calder\'{o}n-Zygmund operator $T$ from $M^{\Phi,\varphi_1}_{w}({\mathbb R}^n)$ to $WM^{\Phi,\varphi_2}_{w}({\mathbb R}^n)$.
\end{cor}

\

\section{Commutators}
Given a function $b$ locally integrable on $\Rn$ and the Calder\'{o}n-Zygmund operator $T$, we consider the linear commutator $[b,T]$ defined by setting, for smooth, compactly supported functions $f$,
$$
[b,T](f) = bT(f) - T(bf).
$$

We recall the definition of the space of $BMO(\Rn)$.

\begin{defn}
Suppose that $b\in L_1^{\rm loc}(\Rn)$, let
\begin{equation*}
\|b\|_\ast=\sup_{x\in\Rn, r>0}\frac{1}{|B(x,r)|} \int_{B(x,r)}|b(y)-b_{B(x,r)}|dy,
\end{equation*}
where
$$
b_{B(x,r)}=\frac{1}{|B(x,r)|} \int_{B(x,r)} b(y)dy.
$$
Define
$$
BMO(\Rn)\equiv BMO=\{ b\in L_1^{\rm loc}(\Rn) ~ : ~ \| b \|_{\ast} < \infty  \}.
$$
\end{defn}

\begin{lem}\cite{S.Janson}
Let $b \in BMO$. Then there is a constant $C>0$ such that
\begin{equation} \label{propBMO}
\left|b_{B(x,r)}-b_{B(x,t)}\right| \le C \|b\|_\ast \ln \frac{t}{r} \;\;\; \mbox{for} \;\;\; 0<2r<t,
\end{equation}
where $C$ is independent of $b$, $x$, $r$ and $t$.
\end{lem}

\begin{lem}\cite{KwokHirMJ}
Let $w\in A_{\infty}$, $b \in BMO$ and $\Phi$ be a Young function with $\Phi\in\Delta_2$. Then,
\begin{equation} \label{Bmo-Worlicz}
\sup_{x\in\Rn, r>0}\Phi^{-1}\big(w(B(x,r))^{-1}\big)
\left\|b-b_{B(x,r)}\right\|_{L^{\Phi}_{w}(B(x,r))}\lesssim \|b\|_\ast.
\end{equation}
\end{lem}

The following result concerning the boundedness of the operator $[b,T]$ on weighted $L^p$ space is known.
\begin{thm}\label{thmwbcomlp}\cite{Alvarezetal}
Let $1<p<\infty$, $w\in A_p$ and $b\in BMO$. Then $[b,T]$ is bounded on $L^{p}_{w}(\Rn)$.
\end{thm}

From this result and \cite[Theorem 2.7]{LiHuYa}, we have the following boundedness of $[b,T]$ on $L^{\Phi}_{w}(\Rn)$.

\begin{thm}\label{thmwbcomorlc}
Let $\Phi$ be a Young function which is of lower type $p_0$ and upper type $p_1$ with $1<p_0\le p_1<\i$, $w\in A_{p_0}$ and $b\in BMO$. Then $[b,T]$ is bounded on $L^{\Phi}_{w}(\Rn)$.
\end{thm}

Additionally, we need the following lemma. For the proof of Lemma \ref{weghtbmo}, see \cite{grafakos} for example.
\begin{lem}\label{weghtbmo}
Let $0 < p < \infty$, $w \in A_{\infty}$ and $b \in BMO$. Then for any ball $B$, we have that
$$
\left(\frac{1}{w(B)} \int_{B}|b(y)-b_{B}|^p w(y)dy\right)^{\frac{1}{p}}\leq C \|b\|_\ast.
$$
\end{lem}

The main result of this section is as follows.
\begin{thm}\label{sufsiocom}
Let $\varphi_1,\varphi_2$ be positive measurable functions satisfying the condition \eqref{es1} and
\begin{equation}\label{wgtcondcom}
\int_{r}^{\infty}\big(1+\ln{\frac{t}{r}}\big)\varphi_1(x,t)\frac{dt}{t}\le C \varphi_2(x,r),
\end{equation}
where $C$ does not depend on $x$ and $r$.
Let also $\Phi$ be a Young function which is of lower type $p_0$ and upper type $p_1$ with $1<p_0\le p_1<\i$, $w\in A_{p_0}$ and $b\in BMO$. Then the commutator of Calder\'{o}n-Zygmund operator $[b,T]$ is bounded from $M^{\Phi,\varphi_1}_{w}({\mathbb R}^n)$ to $M^{\Phi,\varphi_2}_{w}({\mathbb R}^n)$.
\end{thm}

\begin{proof}

For the moment, we denote the commutator on $L^{\Phi}_{w}(\Rn)$ by $[b,T]_{0}$ to avoid confusion.
For $f\in M^{\Phi,\varphi_1}_{w}({\mathbb R}^n)$ and $x\in\Rn$ we choose a ball $B=B(x_0,r)\in \mathcal{B}$ such that $x\in B$, and let
$$
[b,T]f(x):=[b,T]_{0}f_{1}(x)+\int_{\Rn}K(x,y)(b(y)-b(x))f_{2}(y)dy,\, f=f_1+f_2,~f_1=f\chi _{2B}.
$$

First we show that $[b,T]f(x)$ is well-defined $a.e.$ $x$ and independent of the choice $B$ containing $x$.

As $[b,T]_{0}$ is bounded on $L^{\Phi}_{w}(\Rn)$ provided by Theorem \ref{thmwbcomorlc} and $f_1\in L^{\Phi}_{w}(\Rn)$, $[b,T]_{0}f_{1}$ is well-defined.

Next, we show that the second-term of the right-hand side defining $[b,T]f(x)$ converges absolutely for any $f\in M^{\Phi,\varphi_1}_{w}({\mathbb R}^n)$ and almost every $x\in\Rn$.

Observe  that the inclusions   $x\in B$, $y\in \dual{(2B)}$ imply
$\frac{1}{2}|x_0-y|\le |x-y|\le\frac{3}{2}|x_0-y|$.
Then for all $x\in B$
\begin{align*}
&\int_{\Rn}|K(x,y)(b(y)-b(x))f_{2}(y)|dy \lesssim \int_{\dual (2B)} \frac{|b(y)-b(x)|}{|x-y|^{n}}|f(y)| dy
\\
&\lesssim \int_{\dual (2B)} \frac{|b(y)-b(x)|}{|x_0-y|^{n}}|f(y)| dy
\\
&\lesssim \int_{\dual (2B)} \frac{|b(y)-b_{B}|}{|x_0-y|^{n}}|f(y)|dy
+\int_{\dual (2B)} \frac{|b(x)-b_{B}|}{|x_0-y|^{n}}|f(y)|dy
\\
&=I_1+I_2.
\end{align*}
By an argument similar to that used in the estimate (2.25) in \cite{LiNaYaZh}, we have
\begin{equation}\label{esbjm}
\left\||b(\cdot)-b_{B}|w^{-1}(\cdot)\right\|_{L^{\widetilde{\Phi}}_{w}(B)}\lesssim \Phi^{-1}\big(w(B)^{-1}\big)|B|.
\end{equation}
For the sake of completeness, we give the proof the estimate \eqref{esbjm}. Taking into account \eqref{2.3} and Remark \ref{uplowtyp}, it follows that
\begin{align*}
&  \int_{B} \widetilde{\Phi}\big(\frac{|b(x)-b_{B}|w^{-1}(x)}{\Phi^{-1}\big(w(B)^{-1}\big)|B|} \big)w(x)dx  \lesssim \int_{B} \widetilde{\Phi}\big(\frac{|b(x)-b_{B}|\widetilde{\Phi}^{-1}\big(w(B)^{-1}\big)w(B)}{w(x)|B|} \big)w(x)dx
\\
 & \lesssim \frac{1}{w(B)}\int_{B}\left\{\sum_{i=0}^{1}\left[\frac{|b(x)-b_{B}|}{w(x)}\right]^{p_{i}^{\prime}}\left[\frac{w(B)}{|B|}\right]^{p_{i}^{\prime}}\right\}w(x)dx.
\end{align*}
Since $w\in A_{p_0}\subset A_{p_1}$, we know that $w^{1-p_{i}^{\prime}}\in A_{p_{i}^{\prime}}$ for $i\in\{0,1\}$ (see, for example, \cite[p. 136]{Duoandik}). By this, the H\"{o}lder inequality and Lemma \ref{weghtbmo}, we conclude that, for $i\in\{0,1\}$,
\begin{align*}
&\frac{1}{w(B)}\int_{B}|b(x)-b_{B}|^{p_{i}^{\prime}}\left[\frac{w(B)}{|B|}\right]^{p_{i}^{\prime}}\frac{1}{w^{p_{i}^{\prime}}(x)}w(x)dx
\\
&\approx \left[\frac{1}{|B|}\int_{B}w(x)dx\right]^{p_{i}^{\prime}-1}\left[\frac{1}{|B|}\int_{B}w^{1-p_{i}^{\prime}}(x)dx\right]
\\
&~~\times \left\{\frac{1}{[w(B)]^{1-p_{i}^{\prime}}}\int_{B}|b(x)-b_{B}|^{p_{i}^{\prime}}w^{1-p_{i}^{\prime}}(x)dx\right\}\lesssim 1,
\end{align*}
which yields to \eqref{esbjm}.

Now, let us estimate $I_1$.
\begin{align*}
I_1&\thickapprox \int_{\dual
(2B)}|b(y)-b_{B}||f(y)|\int_{|x_0-y|}^{\infty}\frac{dt}{t^{n+1}}dy
\\
&\thickapprox  \int_{2r}^{\infty}\int_{2r\leq |x_0-y|\leq t}
|b(y)-b_{B}||f(y)|dy\frac{dt}{t^{n+1}}
\\
&\lesssim  \int_{2r}^{\infty}\int_{B(x_0,t)}
|b(y)-b_{B}||f(y)|dy\frac{dt}{t^{n+1}}.
\end{align*}
Applying H\"older's inequality, by \eqref{esbjm}, \eqref{propBMO}, \eqref{wgtcondcom} and Lemma \ref{lemHold}  we get
\allowdisplaybreaks
\begin{align}\label{comst1}
I_1 & \lesssim  \int_{2r}^{\infty}\int_{B(x_0,t)}
|b(y)-b_{B(x_0,t)}||f(y)|dy\frac{dt}{t^{n+1}}
\notag \\
&\quad +  \int_{2r}^{\infty}|b_{B(x_0,r)}-b_{B(x_0,t)}|
\int_{B(x_0,t)} |f(y)|dy\frac{dt}{t^{n+1}}
\notag \\
&\lesssim  \int_{2r}^{\infty}
\left\||b(\cdot)-b_{B(x_0,t)}|w^{-1}(\cdot)\right\|_{L^{\widetilde{\Phi}}_{w}(B(x_0,t))} \|f\|_{L^{\Phi}_{w}(B(x_0,t))}\frac{dt}{t^{n+1}}
\notag \\
& \quad +  \int_{2r}^{\infty}|b_{B(x_0,r)}-b_{B(x_0,t)}|
\|f\|_{L^{\Phi}_{w}(B(x_0,t))}\Phi^{-1}\big(w(B(x_0,t))^{-1}\big)\frac{dt}{t}
\notag \\
& \lesssim \|b\|_{*}
\int_{2r}^{\infty}\Big(1+\ln \frac{t}{r}\Big)
\|f\|_{L^{\Phi}_{w}(B(x_0,t))}\Phi^{-1}\big(w(B(x_0,t))^{-1}\big)\frac{dt}{t} \notag \\
&\lesssim \|b\|_{*} \|f\|_{M^{\Phi,\varphi_1}_{w}}
\int_{r}^{\i}\Big(1+\ln \frac{t}{r}\Big)\varphi_1(x_0,t)\frac{dt}{t}\notag \\
&\lesssim \|b\|_{*} \|f\|_{M^{\Phi,\varphi_1}_{w}}\varphi_2(x_0,r)<\infty.
\end{align}
In order to estimate $I_2$ note that $b\in BMO$ implies that $b(\cdot)-b_{B}$ is integrable on $B$, so $b(\cdot)-b_{B}$ is finite almost everywhere on $B$. From this fact, \eqref{sal00} and \eqref{wgtcondcom}, we get
\begin{align}\label{comst2}
I_2&\lesssim |b(x)-b_{B}|
\int_{2r}^{\infty}\|f\|_{L^{\Phi}_{w}(B(x_0,t))}\Phi^{-1}\big(w(B(x_0,t))^{-1}\big)\frac{dt}{t} \notag \\
&\lesssim \|f\|_{M^{\Phi,\varphi_1}_{w}} |b(x)-b_{B}|\int_{r}^{\infty}\varphi_1(x_0,t)\frac{dt}{t} \notag \\
&\lesssim \|f\|_{M^{\Phi,\varphi_1}_{w}} |b(x)-b_{B}|\varphi_2(x_0,r)<\infty.
\end{align}
Similar to the proof of Theorem \ref{sufsio} we can check that the definition is independent of the choice of $B$, the details being omitted.

Now, we show the boundedness.

Since $f_1\in L^{\Phi}_{w}(\Rn)$, by the boundedness of $[b,T]_{0}$ in $L^{\Phi}_{w}(\Rn)$ provided by Theorem \ref{thmwbcomorlc}, it follows that
\begin{equation}\label{es2com}
\|[b,T]f_1\|_{L^{\Phi}_{w}(B)}\leq \|[b,T]f_1\|_{L^{\Phi}_{w}(\Rn)}\lesssim
\|b\|_{*}\|f_1\|_{L^{\Phi}_{w}(\Rn)}=\|b\|_{*}\|f\|_{L^{\Phi}_{w}(2B)}.
\end{equation}
By combining \eqref{es1}, \eqref{es2com} and \eqref{es3}, we get the estimate
\begin{align}\label{es4ftgcom}
& \varphi_2(B)^{-1} \,
\Phi^{-1}\big(w(B)^{-1}\big) \, \|[b,T]_{0}f_1\|_{L^{\Phi}_{w}(B)}  \notag
\\
&\lesssim \|b\|_{*} \varphi_1(2B)^{-1} \,
\Phi^{-1}\big(w(2B)^{-1}\big) \, \|f\|_{L^{\Phi}_{w}(2B)}
\lesssim \|b\|_{*}\|f\|_{M^{\Phi,\varphi_1}_{w}}.
\end{align}

From \eqref{comst1} and \eqref{comst2} for all $x\in B$, we have
\begin{align}\label{gfvjzdhhscom}
  |[b,T]f(x)| &\leq |[b,T]_{0}f_{1}(x)|+\int_{\Rn}|K(x,y)(b(x)-b(y))f_{2}(y)dy| \notag \\
  & \lesssim |[b,T]_{0}f_{1}(x)| + \|b\|_{*} \|f\|_{M^{\Phi,\varphi_1}_{w}}\varphi_2(B)+ \|f\|_{M^{\Phi,\varphi_1}_{w}} |b(x)-b_{B}|\varphi_2(B).
\end{align}
Applying the norm $\|\cdot\|_{L^{\Phi}_{w}}$ on both sides of \eqref{gfvjzdhhscom}, then by \eqref{charorlw} and \eqref{Bmo-Worlicz} we get,
\begin{align*}
\|[b,T]f\|_{L^{\Phi}_{w}(B)}\lesssim& \|[b,T]_{0}f_1\|_{L^{\Phi}_{w}(B)} + \frac{\|b\|_{*} \|f\|_{M^{\Phi,\varphi_1}_{w}}\varphi_2(B)}{\Phi^{-1}\big(w(B)^{-1}\big)}
\\
& + \|f\|_{M^{\Phi,\varphi_1}_{w}} \|b(x)-b_{B}\|_{L^{\Phi}_{w}(B)}\varphi_2(B)
\\
&\lesssim\|[b,T]_{0}f_1\|_{L^{\Phi}_{w}(B)} + \frac{\|b\|_{*} \|f\|_{M^{\Phi,\varphi_1}_{w}}\varphi_2(B)}{\Phi^{-1}\big(w(B)^{-1}\big)} + \frac{\|b\|_{*} \|f\|_{M^{\Phi,\varphi_1}_{w}}\varphi_2(B)}{\Phi^{-1}\big(w(B)^{-1}\big)}
\\
&\approx \|[b,T]_{0}f_1\|_{L^{\Phi}_{w}(B)} + \frac{\|b\|_{*} \|f\|_{M^{\Phi,\varphi_1}_{w}}\varphi_2(B)}{\Phi^{-1}\big(w(B)^{-1}\big)}.
\end{align*}
Consequently, from \eqref{es4ftgcom} we have
\begin{align*}
&\varphi_2(B)^{-1} \,
\Phi^{-1}\big(w(B)^{-1}\big)\|[b,T]f\|_{L^{\Phi}_{w}(B)} \lesssim \varphi_2(B)^{-1} \,
\Phi^{-1}\big(w(B)^{-1}\big)\|[b,T]_{0}f_1\|_{L^{\Phi}_{w}(B)}
\\
& + \|b\|_{*}\|f\|_{M^{\Phi,\varphi_1}_{w}}  \lesssim \|b\|_{*}\|f\|_{M^{\Phi,\varphi_1}_{w}}.
\end{align*}
By taking supremum over $B\in\mathcal{B}$, we obtain the boundedness of $[b,T]$ from $M^{\Phi,\varphi_1}_{w}({\mathbb R}^n)$ to $M^{\Phi,\varphi_2}_{w}({\mathbb R}^n)$.
\end{proof}

If we take $\Phi(t)=t^p,~1\le p<\infty$ at Theorem \ref{sufsiocom} we get the following corollary for generalized weighted Morrey spaces which was proved in \cite{GulKarMustSer}.
\begin{cor}
Let $\varphi_1,\varphi_2$ be positive measurable functions satisfying the condition \eqref{es1} and \eqref{wgtcondcom}. Let also $1<p<\infty$, $w\in A_{p}$ and $b\in BMO$. Then the commutator of Calder\'{o}n-Zygmund operator $[b,T]$ is bounded from $M^{p,\varphi_1}_{w}({\mathbb R}^n)$ to $M^{p,\varphi_2}_{w}({\mathbb R}^n)$.
\end{cor}

\

\section{Vector-valued inequality for Calder\'{o}n-Zygmund operators}
In this section we shall consider the vector-valued inequality for Calder\'{o}n-Zygmund operators on generalized weighted Orlicz-Morrey spaces.

\begin{defn}
Let $\varphi$ be a positive measurable function
on $\Rn\times (0,\i)$, $w$ be a non-negative measurable function on $\Rn$, $\Phi$ be any Young function and $1\le q \le \infty$. The generalized vector-valued weighted Orlicz-Morrey spaces $M^{\Phi,\varphi}_{w}(l_{q})=M^{\Phi,\varphi}_{w}(l_{q},\Rn)$ is defined as the set of all sequences
$F=\{ f_j \}_{j=1}^{\infty}$ of Lebesgue measurable functions on $\Rn$ such that
$$
\left\| F \right\|_{M^{\Phi,\varphi}_{w}(l_{q})}=\left\| \{ f_j \}_{j=1}^{\infty} \right\|_{M^{\Phi,\varphi}_{w}(l_{q})} : =
\left\| \left\| \{ f_j(\cdot) \}_{j=1}^{\infty} \right\|_{l_{q}}\right\|_{M^{\Phi,\varphi}_{w}}<\infty.
$$
\end{defn}

The proof of the following vector-valued modular inequality for Calder\'{o}n-Zygmund operator $T$ in weighted Orlicz spaces can be found in \cite[Theorem 5.1]{GogKrb1995}.
\begin{prop}\label{VecMaxweigOrlcold}
Let $1<q<\infty$, $\Phi$ be a Young function with $\Phi\in\Delta_2\cap\nabla_2$.
Assume in addition $w \in A_{i_{\Phi}}$. Then, for any family of locally integrable functions $F=\{ f_j \}_{j=1}^{\infty}$,
\begin{equation*}
\int_{\Rn}\Phi\left(\|TF(x)\|_{l_q}\right)w(x)dx\leq C \int_{\Rn}\Phi\left(\|F(x)\|_{l_q}\right)w(x)dx
\end{equation*}
for some $C>0$ independent of $F$, where $TF=\{ T f_j \}_{j=1}^{\infty}$.
\end{prop}

The proof of the following lemma can be found in \cite[Lemma 6]{HakNakSaw}.
\begin{lem}\label{lem:131205-199}
For any ball $B$,
we have
$$
\int_{\mathbb{R}^n\backslash 2B}|K(x,y)f(y)|dy
\lesssim
\sum\limits_{k=1}^{\infty}
\frac{1}{|2^{k+1} B|} \int_{2^{k+1} B}
|f(y)| dy
$$
for all $x\in B$.
\end{lem}

\begin{thm} \label{3.4.VectV}
Let $1<q<\infty$, $\Phi$ be a Young function with $\Phi\in\Delta_2\cap\nabla_2$, $w \in A_{i_{\Phi}}$ and $(\varphi_1,\varphi_2)$ satisfy the conditions \eqref{es1} and
\eqref{wgtcond}. Then the operator $T$ is bounded from $M^{\Phi,\varphi_1}_{w}(l_{q})$ to $M^{\Phi,\varphi_2}_{w}(l_{q})$,
i.e., there is a constant $C>0$ such that
\begin{equation}\label{vvinqinf}
\|T F\|_{M^{\Phi,\varphi_2}_{w}(l_{q})} \le C \|F\|_{M^{\Phi,\varphi_1}_{w}(l_{q})}
\end{equation}
holds for all $F=\{ f_j\}_{j=1}^{\infty} \in M^{\Phi,\varphi_1}_{w}(l_{q})$.
\end{thm}

\begin{proof}
We split $F=\{f_j\}_{j=1}^{\i}$ with
\begin{align*}
&F=F_1+F_2, ~~~~
F_1=\{ f_{j,1}\}_{j=1}^{\infty}, ~~~~
F_2=\{ f_{j,2}\}_{j=1}^{\i},
\\
&f_{j,1}(y)=f_j(y)\chi_{B(z,2r)}(y),\quad
 f_{j,2}(y)=f_j(y)\chi_{\Rn\backslash B(z,2r)}(y),
\quad r>0. \notag
\end{align*}

Let $x \in B=B(z,r)$ be fixed. Inspired by the ideas of \cite{HakNakSaw}, from Lemma \ref{lem:131205-199} and $\ell^{q}-\ell^{q^{\prime}}$ duality, we have
\begin{equation*}
\begin{split}
\|TF_2(x)\|_{l_q} & \lesssim \left(\sum\limits_{j=1}^{\infty}\left(\sum\limits_{k=1}^{\infty}
\frac{1}{|2^{k+1} B|}\int_{2^{k+1} B}|f_j(y)|dy\right)^q\right)^{1/q}\\
&=\sum\limits_{j=1}^{\infty}\sum\limits_{k=1}^{\infty}
\frac{a_j}{|2^{k+1} B|}\int_{2^{k+1} B}|f_j(y)|dy,
\end{split}
\end{equation*}
where $\{a_j\}_{j=1}^{\infty}$ is an $\ell^{q^{\prime}}$ sequence with $\|a_j\|_{l^{q^{\prime}}}=1$.

We use H\"{o}lder's inequality to obtain
\begin{equation*}
\begin{split}
\|TF_2(x)\|_{l_q} & \lesssim \sum\limits_{k=1}^{\infty}
\frac{1}{|2^{k+1} B|}\int_{2^{k+1} B}\sum\limits_{j=1}^{\infty} a_j|f_j(y)|dy\\
&\lesssim \sum\limits_{k=1}^{\infty}
\frac{1}{|2^{k+1} B|}\int_{2^{k+1} B}\|f_j(y)\|_{l^q}dy\\
&\lesssim \int_{2r}^\infty \frac{\|\;\|F(\cdot)\|_{l_q}\|_{L^1(B(z,t))}}{t^{n+1}}\,dt.
\end{split}
\end{equation*}
Hence by Lemma \ref{lemHold} and \eqref{wgtcond}, we get for all $x\in B$
\begin{align}\label{vingrep2}
\|TF_2(x)\|_{l_q} & \lesssim \int_{2r}^\infty \Phi^{-1}(w(B(z,t))^{-1}) \|\;\|F\|_{l_q}\|_{L^{\Phi}_{w}(B(z,t))}\frac{dt}{t}\notag\\
&\lesssim \|F\|_{M^{\Phi,\varphi_1}_{w}(l_{q})}\int_{r}^\infty \varphi_{1}(z,t)\frac{dt}{t}\notag\\
&\lesssim \|F\|_{M^{\Phi,\varphi_1}_{w}(l_{q})}\varphi_{2}(z,r).
\end{align}

By Proposition \ref{VecMaxweigOrlcold} we have
\begin{align} \label{GAzel1}
\|\;\|T F_1\|_{l_q}\|_{L^{\Phi}_{w}(B)} & \le
\|\;\|T F_1\|_{l_q}\|_{L^{\Phi}_{w}} \notag
\\
& \le C \|\;\|F_1\|_{l_q}\|_{L^{\Phi}_{w}} = C \|\;\|F\|_{l_q}\|_{L^{\Phi}_{w}(2B)},
\end{align}
where $C>0$ is independent of the vector-valued function $F$.

By combining \eqref{es1}, \eqref{es3} and \eqref{GAzel1}, we get the estimate
\begin{align}\label{es4ftgvv}
& \varphi_2(B)^{-1} \,
\Phi^{-1}\big(w(B)^{-1}\big) \, \|\;\|T F_1\|_{l_q}\|_{L^{\Phi}_{w}(B)} \notag
\\
& \lesssim \varphi_1(2B)^{-1} \,
\Phi^{-1}\big(w(2B)^{-1}\big) \, \|\;\|F\|_{l_q}\|_{L^{\Phi}_{w}(2B)}
\lesssim\|F\|_{M^{\Phi,\varphi_1}_{w}(l_{q})}.
\end{align}

From \eqref{vingrep2} for all $x\in B$, we have
\begin{align}\label{gfvjzdhhsvv}
  \|T F(x)\|_{l_q} &\leq \|T F_1(x)\|_{l_q}+\|TF_2(x)\|_{l_q} \notag \\
  & \lesssim \|T F_1(x)\|_{l_q} + \|F\|_{M^{\Phi,\varphi_1}_{w}(l_{q})}\,\varphi_2(B).
\end{align}
Applying the norm $\|\cdot\|_{L^{\Phi}_{w}}$ on both sides of \eqref{gfvjzdhhsvv}, then by \eqref{charorlw} we get,
\begin{equation*}
\|\;\|T F\|_{l_q}\|_{L^{\Phi}_{w}(B)}\lesssim \|\;\|T F_1\|_{l_q}\|_{L^{\Phi}_{w}(B)} + \frac{\|F\|_{M^{\Phi,\varphi_1}_{w}(l_{q})}\,\varphi_2(B)}{\Phi^{-1}\big(w(B)^{-1}\big)}.
\end{equation*}
Consequently, by \eqref{es4ftgvv} we have
\begin{align*}
&\varphi_2(B)^{-1} \,
\Phi^{-1}\big(w(B)^{-1}\big)\|\;\|T F\|_{l_q}\|_{L^{\Phi}_{w}(B)} \notag \\
&~~~~\lesssim \varphi_2(B)^{-1} \,
\Phi^{-1}\big(w(B)^{-1}\big)\|\;\|T F_1\|_{l_q}\|_{L^{\Phi}_{w}(B)}+ \|F\|_{M^{\Phi,\varphi_1}_{w}(l_{q})}  \lesssim \|F\|_{M^{\Phi,\varphi_1}_{w}(l_{q})}.
\end{align*}
By taking supremum over $B\in\mathcal{B}$, we obtain the boundedness of $T$ from $M^{\Phi,\varphi_1}_{w}(l_{q})$ to $M^{\Phi,\varphi_2}_{w}(l_{q})$.
\end{proof}

\

\end{document}